\documentclass[11pt,epsf]{article}
\usepackage{graphicx}
\usepackage{amsthm}
\usepackage{amsfonts}
\usepackage{amssymb}
\title{A domain free of the zeros of the partial theta function}
\author{Vladimir Petrov Kostov\\
  Universit\'e de Nice, 
Laboratoire de Math\'ematiques, Parc Valrose,\\ 06108 Nice Cedex 2, France,  
e-mail: vladimir.kostov@unice.fr} 
\date{}
\newtheorem{tm}{Theorem}

\newtheorem{rem}[tm]{Remark}

\newtheorem{lm}[tm]{Lemma}

\newtheorem{prop}[tm]{Proposition}
\newtheorem{nota}[tm]{Notation}

\begin{document}

\maketitle 

\begin{abstract}
  We prove that for $q\in (0,1)$, the partial theta function
  $\theta (q,x):=\sum _{j=0}^{\infty}q^{j(j+1)/2}x^j$ has no zeros in the closed
  domain $\{ \{ |x|\leq 3\} \cap \{${\rm Re}$x\leq 0\}
  \cap \{ |${\rm Im}$x|\leq 3/\sqrt{2}\} \} \subset \mathbb{C}$ and no real zeros $\geq -5$.


\end{abstract}

\section{Introduction}

The present paper deals with analytic properties of the
{\em partial theta function}

\begin{equation}\label{eqtheta}
  \theta (q,x):=\sum _{j=0}^{\infty}q^{j(j+1)/2}x^j~.
  \end{equation}
It owes its name to the resemblance between the function 
$\theta (q^2,x/q)=\sum _{j=0}^{\infty}q^{j^2}x^j$ and the
{\em Jacobi theta function} 
$\Theta (q,x):=\sum _{j=-\infty}^{\infty}q^{j^2}x^j$; ``partial'' refers to the fact
that summation in the case of $\theta$ takes place only from $0$ to $\infty$.

We consider the situation when the variable $x$ and the parameter $q$ are real,
more precisely, when $(q,x)\in (0,1)\times \mathbb{R}$. This function
has been studied also for $(q,x)\in (-1,0)\times \mathbb{R}$
and $(q,x)\in \mathbb{D}_1\times \mathbb{C}$; here $\mathbb{D}_1$ stands for
the open unit disk. For any fixed non-zero value of the parameter $q$
($|q|<1$),
the function
$\theta (q,.)$ is an entire function in $x$ of order~$0$.

The partial theta function finds various applications -- from
Ramanujan type $q$-series 
(\cite{Wa}) to the theory 
of (mock) modular forms (\cite{BrFoRh}), from asymptotic analysis (\cite{BeKi}) 
to statistical physics 
and combinatorics (\cite{So}). How $\theta$ can be applied to problems
dealing with asymptotics and modularity of partial and false theta 
functions and their relationship to representation theory and conformal field 
theory is made clear in \cite{CMW} and \cite{BFM}. The place which this
function finds in Ramanujan's lost notebook is explained
in~\cite{AnBe} and~\cite{Wa}.
Its Pad\'e approximants are studied in~\cite{LuSa}.

A recent interest in the partial theta function is connected with the study of 
section-hyperbolic polynomials, i.~e. real polynomials with positive coefficients, with
all roots real negative and all whose finite sections (i.e. truncations) 
have also this property, see \cite{KoSh}, \cite{KaLoVi} and \cite{Ost}; 
the cited papers use results of Hardy, Petrovitch and Hutchinson 
(see \cite{Ha}, \cite{Pe} and \cite{Hu}). Various analytic properties of the
partial theta function are proved in \cite{Ko2}-\cite{Ko13} and other papers
of the author.

The analytic properties of $\theta$ known up to now, in particular, the
behaviour of its zeros, are discussed in the next section. One of them
is the fact that for any $q\in (0,1)$, all complex conjugate
pairs of zeros of $\theta (q,.)$ remain within the domain

$$\{ {\rm Re}x\in (-5792.7, 0), |{\rm Im}x|<
132\} \cup \{ |x|<18\} ~.$$
For $q\in (-1,0)$, this is true for the domain
$\{ |${\rm Re}$x|<364.2, |${\rm Im}$x|<132\}$, see~\cite{Ko10} and~\cite{Ko8}.
In this sense the complex conjugate zeros of $\theta$ never go too far from the
origin. It is also true that they never enter into the unit disk,
see~\cite{Ko13} (but this property is false if $q$ and $x$ are complex, see
the next section). In the present paper we exhibit a convex domain
which contains the left unit half-disk, which is more than
$7$ times larger than the latter and which is free of zeros
of $\theta$ for any $q\in (0,1)$:

\begin{tm}\label{tmmain}
  For any fixed $q\in (0,1)$, the partial theta function has no zeros
  in the domain
 $\mathcal{D}:=\{ \{ |x|\leq 3\} \cap \{${\rm Re}$x\leq 0\}
  \cap \{ |${\rm Im}$x|\leq 3/\sqrt{2}
  \} \} \subset \mathbb{C}$ (with $3/\sqrt{2}=2.121320344\ldots$). 
  \end{tm}

When only the real zeros of $\theta$
are dealt with, one can improve the above theorem: 

\begin{prop}\label{propmain}
  For any $q\in (0,1)$ fixed, the function $\theta (q,.)$ has no
  real zeros~$\geq -5$.
\end{prop}

Before giving comments on these results in the next section we explain
the structure of the paper. Section~\ref{secremind}
reminds certain analytic properties of $\theta$. Proposition~\ref{propmain}
is proved in Section~\ref{prpropmain}. In Section~\ref{secprelim} we prove 
some lemmas which are used to prove Theorem~\ref{tmmain}; their proofs
can be skipped at first reading. Section~\ref{secplan} contains a plan
of the proof of Theorem~\ref{tmmain}. The proofs of the proposition and lemmas
formulated in Section~\ref{secplan} can be found in
Section~\ref{secproofs}.

\section{Comments\protect\label{seccomments}}

Throughout the paper we use the following notation:

\begin{nota}\label{notat}
  {\rm We define four arcs of the circle centered at $0$ and of radius~$3$:}

$$C_k:=\{ x\in \mathbb{C}||x|=3,
    {\rm arg}x\in [\pi /2+(k-1)\pi /4,\pi /2+k\pi /4]\} ~,~~~k=1,~2,~3,~4~.$$
{\rm We set $w:=3/\sqrt{2}=2.121320344\ldots$. The border
$\partial \mathcal{D}$ of the domain $\mathcal{D}$ defined in
Theorem~\ref{tmmain} consists
of the arc $C_2\cup C_3$, the horizontal segments $S_{\pm}:=[-w\pm wi,\pm wi]$
and the vertical segment $S_v:=[-wi,wi]$. We parametrise the segment
$S_+$ by setting $x:=-t+wi$, $t\in [0,w]$.}
\end{nota}

One can make the following observations with regard to Theorem~\ref{tmmain}
and Proposition~\ref{tmmain}:
\vspace{1mm}

(1) It is not clear whether Theorem~\ref{tmmain} should hold true for
    the whole of the left half-disk of radius $3$, because 
    $|\theta (0.71,e^{0.5188451144\pi i})|=0.0141\ldots$, i.~e. one obtains
    a very small value of $|\theta |$ for a point of the arc $C_1$. This
    might mean that a zero of $\theta$ crosses the arc $C_1$ for $q$ close to
    $0.71$.
    \vspace{1mm}
    
    (2) The difficulty to prove results as the ones of Theorem~\ref{tmmain} and
    Proposition~\ref{propmain} resides in the fact that the
    rate of convergence of the series of $\theta$ decreases as $q$
    tends to $1^-$,
    and for $q=1$, one obtains as limit of $\theta$ the rational (not entire)
    function $1/(1-x)$.
    It is true that the series
    of $\theta$ converges to the function $1/(1-x)$ (which has no zeros at all)
    on a domain larger than the unit disk
    and containing the domain $\mathcal{D}$, see~\cite{Ka}. Yet one disposes
    of no concrete estimations about this convergence, so one cannot deduce
    from it the absence of zeros of $\theta$ in the domain~$\mathcal{D}$ for
    all $q\in (0,1)$.
    \vspace{1mm}
    
    (3) The domain $\mathcal{D}$ contains the left half-disk of radius
    $3/\sqrt{2}>2$. The ratio of the surfaces of $\mathcal{D}$ and of the
    left unit half-disk is
    $(\pi 3^2/4+(3/\sqrt{2})^2)/(\pi /2)=7.364788974\ldots$.
    \vspace{1mm}
    
    (4) One knows that for $q=0.3092\ldots$, the function $\theta (q,.)$
    has a double
real zero $-7.5032\ldots$, see~\cite{KoSh}. Pictures of the zero set of the
function $\theta$ (see \cite{Ko12}) suggest that for certain values of
$q\in (0,1)$, it has
a zero in the interval $(-7,-6)$, so Proposition~\ref{propmain} cannot be
made much stronger.
\vspace{1mm}

We explain by examples why analogs of
the property of the zeros of $\theta$ to avoid
the domain $\mathcal{D}$ cannot be found in cases other than
$q\in (0,1)$, $x\leq 0$:
\vspace{1mm}

(i) If $q$ is complex, then some of the zeros of $\theta$ can be of modulus
$<1$. Indeed, for $q=\rho e^{3\pi i/4}$, where $\rho \in (0,1)$
is close to $1$, the function $\theta$ has a zero close to
$0.33\ldots +0.44\ldots i$
whose modulus is $0.56\ldots <1$. Similar examples 
can be given for any $q$ of the form
$\rho e^{k\pi i/\ell}$, $k$, $\ell \in \mathbb{Z}^*$, see~\cite{Ko13}.
It is true however that $\theta$ has no zeros for $|x|\leq 1/2|q|$, see
Proposition~7 in~\cite{Ko1}.
\vspace{1mm}

(ii) If $q\in (0,1)$, the function $\theta$
has no positive zeros, but $\theta (0.98,.)$ is likely to have a zero close
to $1.209\ldots+0.511\ldots i$ (i.~e. of modulus $1.312\ldots$),
see~\cite{Ko13}. Conjecture: {\em As $q\rightarrow 1^-$, one can find complex
  zeros of $\theta (q,.)$ as close to $1$ as possible.} One can check
numerically that for $q$ close to $0.726475$, $\theta$ has a complex
conjugate couple of zeros close to $\pm 2.9083\ldots i$ (which by the way
corroborates the idea that the statement of Theorem~\ref{tmmain} cannot be
extended to the whole of the left half-disk of radius $3$).   
%and purely imaginary
  %couples of zeros as close to $\pm e^{\pi /2}i=\pm 4.810\ldots i$
%as possible.}
  Thus a convex domain free of zeros of $\theta$ should belong to the
rectangle $\{ {\rm Re}x\in (0,1)$, $|{\rm Im}x|<2.9083\ldots \}$.
\vspace{1mm}

(iii) For
$q\in (-1,0)$, it is true that the leftmost of the positive zeros of $\theta$
tends to $1^+$ as $q$ tends to $-1^+$, see part (2) of
Theorem~3 in~\cite{Ko12}. The function
$\theta (-0.96,.)$ is supposed to have a couple of conjugate zeros
close to the zeros
$z_{\pm}:=0.824\ldots \pm 1.226\ldots i$ (of modulus $1.478\ldots$)
of its truncation
$\theta _{100}^{\bullet}(-0.96,.)$; when truncating, the first two skipped terms
are of modulus
$6.57\ldots \times 10^{-75}$ and $1.51\ldots \times 10^{-76}$.
As $q\rightarrow -1^+$, the limit of $\theta$ equals $(1-x)/(1+x^2)$.
One can suppose that the zeros, which equal $z_{\pm}$ for $q=-0.96$,
tend to~$\pm i$
as $q\rightarrow -1^+$. One knows that for $q\in (-1,0)$,
complex zeros do not cross the imaginary axis, see Theorem~8 in~\cite{Ko12}.
Hence these zeros of $\theta$ should remain in the right half-plane and close
to~$\pm i$. This means that it is hard to imagine a convex domain in the right
half-plane much larger than the right unit half-disk and free of zeros of
$\theta$.

As for the left half-plane, the truncation $\theta _{100}^{\bullet}(-0.96,.)$ of
$\theta (-0.96,.)$ has conjugate zeros $0.769\ldots \pm 1.255\ldots i$
(of modulus $1.473\ldots$) about which, as about $z_{\pm}$ above,
one can suggest that they tend to $\pm i$ as $q\rightarrow -1^+$. 
This could make one think that if one wants to find a
domain in the left half-plane containing the left unit half-disk and free of
zeros of $\theta$, then in this domain 
the modulus of the imaginary part should not be larger than $1$.
On the other hand $\theta (-0.7,.)$ has a zero close to $w_0:=-2.69998\ldots$
%and the 
%negative double zeros of $\theta$ tend to
%$-e^{\pi /2}=-4.810\ldots$ as $q\rightarrow -1^+$ (see \cite{Ko5}),
so the width of
the desired domain should be $<|w_0|$.

\section{Known properties of the partial theta function
  \protect\label{secremind}}

In this section we remind first that the Jacobi theta function satisfies the
{\em Jacobi triple product}

$$\sum _{j=-\infty}^{\infty}q^{j^2}x^{2j}=
\Theta (q,x^2)=\prod _{m=1}^{\infty}(1-q^{2m})(1+x^2q^{2m-1})(1+x^{-2}q^{2m-1})$$
from which we deduce the equalities

\begin{equation}\label{eqtriple}
  \begin{array}{rcl}
    \Theta ^*(q,x):=\Theta (\sqrt{q},\sqrt{q}x)&=&
  \sum _{j=-\infty}^{\infty}q^{j(j+1)/2}x^j=
  \prod _{m=1}^{\infty}(1-q^m)(1+xq^m)(1+q^{m-1}/x)\\ \\ 
  &=&(1+1/x)\prod_{m=1}^{\infty}((1-q^m)(1+xq^m)(1+q^m/x))~.\end{array}
  \end{equation}
It is clear that

\begin{equation}\label{eqthetaG}
  \theta =\Theta ^*-G~~~\, {\rm with}~~~\,
  G(q,x):=\sum _{j=-\infty}^{-1}q^{j(j+1)/2}x^j~.
\end{equation}

\begin{nota}\label{notaxXt}
  {\rm (1) When treating the function $G$ we often change the variable $x$ to
    $X:=1/x$. To distinguish the truncations of the function $\theta$ in
    the variable $x$ from the ones in the variable $t$
    (see Notation~\ref{notat}) we write
    $\theta =\theta _k^{\bullet}+\theta _*^{\bullet}$, where
    $\theta _k^{\bullet}:=\sum _{j=0}^kq^{j(j+1)/2}x^j$ and
    $\theta _*^{\bullet}:=\sum _{j=k+1}^{\infty}q^{j(j+1)/2}x^j$, i.~e. we use the
    superscript ``bullet'' when in the variable $x$ (no index $k$ is added to $\theta _*^{\bullet}$).
    No superscript is used for the truncations
    of $\theta (q,-t+wi)$ and of $G$.
    
    (2) We set $\lambda :=3e^{3\pi i/4}$, $R(q,x):=\prod _{m=1}^{\infty}(1+q^{m-1}/x)$, $M:=|(1+qx)(1+q/x)|$, $M_0:=(1-q)M$ and $M_1(q,t):=M_0(q,-t+wi)$.}
  \end{nota}

\begin{rem}
  {\rm In the proofs we use the convergence of the series (\ref{eqtheta})
    when the
    parameter $q$ belongs to an interval of the form $[0,a]$, $a\in (0,1)$.
    When we need to deal with intervals of the form $[a,1]$, we use the
  equalities (\ref{eqthetaG}) in which the modulus of
  the term $\Theta ^*$ tends to $0$ as $q$
  tends to $1^-$ while the series of $G$ converges uniformly for
  $|x|\in [c,\infty )$ for any fixed $c>1$. When in the proof of a lemma or a proposition we use the fact that a certain function in one variable (mainly a polynomial) is increasing or decreasing, we do not give a detailed proof of this, because in all such cases the proof can be given using elementary methods (computation of derivatives and numerical computation of their real roots). We do not give details when proving the absence of critical points of polynomials in two variables in given rectangles. In this text their degree is never too high and the necessary computations are easily performed using MAPLE.}
  \end{rem}

For $q\in (0,1)$, the real zeros of $\theta$ (which are all negative)
and of any of its derivatives
w.r.t. the variable $x$ form a sequence tending to $-\infty$ and
behaving asymptotically as a geometric progression with ratio~$1/q$,
see Theorem~4 in~\cite{Ko1}.

There exists an increasing and tending to $1^-$ sequence of
{\em spectral values} 
$\tilde{q}_j$ of $q$ such that $\theta (\tilde{q}_j,.)$ has a multiple
(more exactly double) real zero, see~\cite{KoSh}.
The $6$-digit truncations of the first $6$
spectral values are:  

$$0.309249~,~~~0.516959~,~~~0.630628~,~~~
0.701265~,~~~0.749269~,~~~0.783984~.$$
When $q$ passes from $\tilde{q}_j^-$ to $\tilde{q}_j^+$, the rightmost two of
the real zeros of $\theta$ coalesce and then form a complex conjugate pair.
All other real zeros of $\theta$ remain negative and distinct, see Theorem~1
in~\cite{Ko1}. The inverse
(complex couples becoming double and then two distinct real zeros) never
happens. No zeros are born at~$\infty$.

Thus for $q$ fixed, the function $\theta$ belongs to the Laguerre-P\'olya
class $\mathcal{L-P}I$ exactly if $q\in (0,\tilde{q}_1]$. For
  $q\in (\tilde{q}_j,\tilde{q}_{j+1}]$, the function $\theta$ is the product of
  a real polynomial of degree $2j$ without real zeros and a function of the
  class $\mathcal{L-P}I$. See the details in~\cite{Ko1A}.

  Spectral values exist also for $q\in (-1,0)$, see~\cite{Ko5}. The existence
  of spectral values for complex values of $q$ is proved in~\cite{Ko5}, see
  Proposition~8 therein.

  At the end of this section we mention the fact that the function
  $\theta$ satisfies the two conditions 

  $$\theta (q,x)=1+qx\theta (q,qx)~~~\, {\rm and}~~~\,
  2q\partial \theta /\partial q=2x\partial \theta /\partial x+
  x^2\partial ^2\theta /\partial x^2~.$$

\section{Proof of Proposition~\protect\ref{propmain}\protect\label{prpropmain}}

 For $q\leq 0.1$, all zeros of $\theta (q,.)$ are real negative and smaller than
 $-1/q$, see Proposition~7 in~\cite{Ko1}. Hence they are smaller than $-5$ and
 one has $\theta (q,x)>0$ for $x\in [-1/q,0]\supset [-5,0]$.
 As $q$ increases, its zeros
  depend continuously on $q$. For a spectral value of $q$, 
  certain zeros coalesce to form then
  a complex conjugate pair, but new real zeros are not born,
  see the previous section. Therefore it
  suffices to show that for $q\in (0,1)$, one has $\theta (q,-5)>0$.

  For $q\in (0,0.8]$, one finds numerically that $\theta (q,-5)$ is larger
          than $0.04$. To this end one can consider
          $\theta _{15}^{\bullet}(q,-5):=\sum _{j=0}^{15}q^{j(j+1)/2}(-5)^j$
          which is a polynomial
          in $q$ and show that for $q\in (0,0.5]$ (resp. for $q\in [0.5,0.8]$),
            it is larger than $0.05$ (resp. larger than $0.16$).
            The sum of the moduli
        of all skipped terms is smaller than $1.8\times 10^{-30}$ and
        $0.07$ respectively.

        To prove that $\theta (q,-5)>0$ for $q\in [0.8,1)$, we first show that
          $-G(q,-5)>4/25$. This is true, because $-G(q,-5)$ is a Leibniz series
          with first two terms $1/5$ and $-q/25$, so its sum is
          $\geq 1/5-q/25>1/5-1/25=4/25$.

          Now we estimate $|\Theta ^*|$. We set

          $$K(q):=(1+qx)(1+q/x)|_{x=-5}=(1-5q)(1-q/5)=1-26q/5+q^2~.$$
          Hence $\Theta ^*(q,-5)=(4/5)\prod _{m=1}^{\infty}(1-q^m)
          \prod _{m=1}^{\infty}K(q^m)$, see (\ref{eqtriple}).
          The graph of the function $|K|$
          (for $q\in [0,1]$) is shown in
          Fig.~\ref{figmodulusK}.

          \begin{figure}[htbp]
%\includegraphics[scale=0.5]{parthetanegfirstfour.eps}
%\centerline{\hbox{\includegraphics[scale=0.7]{parthetanegfirstfour.eps}}}
\centerline{\hbox{\includegraphics[scale=0.7]{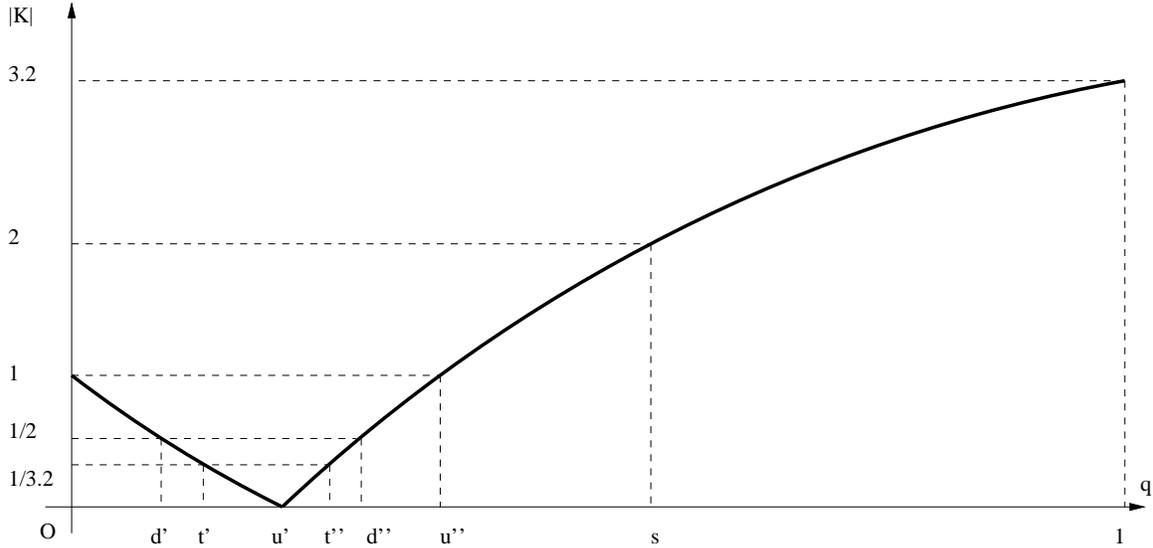}}}
%\centerline{\hbox{\epsfxsize=10cm \epsfbox{k=1234.pdf}}}
\caption{The graph of the function $|K|:=|1-26q/5+q^2|$.}
\label{figmodulusK}
          \end{figure}
          The function $|K|$ is decreasing on $[0,u']$ and increasing on
          $[u',1]$, where $u'=0.2$. One has $|K(0)|=|K(u'')|=1$,
          $u''=0.4182575771\ldots$, and $|K(1)|=3.2$.
          In Fig.~\ref{figmodulusK} we show the values
          $0.1357556939\ldots =:t'<t'':=0.2660119966\ldots$ of $q$ for which
          $|K|=1/3.2$, the values
          $0.09800079936\ldots =:d'<d'':
          =0.3065310118\ldots$ where
          $|K(q)|=1/2$ and the value $s:
          =0.6609280570\ldots$ of $q$
          such that $|K(q)|=2$.

          For $a\in (0,1]$, set $\phi (a):=\ln (1/a)/\ln (1/q)$, so
        $\phi (a)\geq 0$ and $\phi (a_1)-\phi (a_2)=\phi (a_1/a_2)$.
        We remind that
        
        \begin{equation}\label{eq[]}
          {\rm for}~~~\, b, c\in \mathbb{R}~,~~~\, 
        [b-c]=[b]-[c]-\chi (b,c)~,~~~\, {\rm where}~~~\, \chi (b,c)=0
        ~~~\, {\rm or}~~~\, 1~;\end{equation}
        here $[.]$ stands for the
        integer part. Denote by $\ell _1$, $\ell _2\in \mathbb{N}$
        the maximal indices
          $\ell$ for which
          one has $q^{\ell}\geq u''$ and $q^{\ell}\geq s$
          respectively:

          $$\ell _1=[\phi (u'')]~~~\, {\rm and}~~~\, \ell _2=[\phi (s)]~.$$ 
          These are the numbers of terms of the sequence $\{ q^m\}$ belonging
          to the intervals $[u'',1)$ and $[s,1)$. For the intervals
          $[t',t'')$ and $[d',d'')$, these numbers are equal to

          $$m_1:=[\phi (t')]-[\phi (t'')]~~~\, {\rm and}~~~\,
          m_2:=[\phi (d')]-[\phi (d'')]~.$$
          One computes directly that

          $$\ln (t''/t')=0.672\ldots >0.414\ldots =\ln (1/s)~,~~~\,
          {\rm thus}~~~\, \ln (t''/t')-\ln (1/s)>0.25~.$$
          Moreover, as $q\in [0.8,1)$, one has
            $1/\ln (q)\geq 1/\ln (0.8)=4.48\ldots >4$, so

            $$\begin{array}{l}
              \ln (t''/t')/\ln (1/q)-\ln (1/s)/\ln (1/q)>1~~~\, {\rm hence}~~~\,
                  [\phi (t'/t'')]>[\phi (s)]~~~\, {\rm and}\\ \\
                  m_1=[\phi (t')]-[\phi (t'')]=[\phi (t'/t'')]+
                  \chi (\phi (t'),\phi (t''))
                  >[\phi (s)]=\ell _2~.\end{array}$$
            This means that out of the factors
            $|K(q^m)|$ which are present in $|\Theta ^*(q,-5)|$,
            the ones which are in the interval $[2,3.2)$ are less than
              the ones which are in $[0,1/3.2]$. We denote their sets by
              $\Sigma _{[2,3.2)}$ and $\Sigma _{[0,1/3.2]}$ and we write
                $\sharp (\Sigma _{[2,3.2)})<\sharp (\Sigma _{[0,1/3.2]})$.
                  The latter inequality implies 

                  $$\prod _{|K(q^m)|\in \Sigma _{[2,3.2)}\cup \Sigma _{[0,1/3.2]}}|K(q^m)|<1~.$$
Using the above notation and (\ref{eq[]}) one can write

\begin{equation}\label{eqphi}
  \begin{array}{lllll}
                    \sharp (\Sigma _{[1,2)})&=:&\ell _3&=&[\phi (u'')]-
                      [\phi (s)]\leq [\phi (u''/s)]+1~~~\, {\rm and}\\ \\
                      \sharp (\Sigma _{(1/3.2,1/2]})&=:&m_3&\geq &[\phi (d')]-
                           [\phi (t')]+[\phi (t'')]-
                           [\phi (d'')]-1~.\end{array}
  \end{equation}
                  To prove the latter inequality one has to observe
                  that the numbers $q^m$
                  corresponding to factors $|K(q^m)|$ from the set
                  $\Sigma _{(1/3.2,1/2]}$ belong to the union
                $[d',t')\cup (t'',d'']$. The numbers $q^m$ which belong to the
                interval 
            $[d',t')$ are exactly $[\phi (d')]-[\phi (t')]$. The ones that
              are in $(t'',d'']$ are not less than
            $[\phi (t'')]-[\phi (d'')]-1$ (at most one of them 
            equals $t''$ and there are exactly
            $[\phi (t'')]-[\phi (d'')]$ numbers
            $q^m$ in $[t'',d'')$). One finds that 
$$
  \phi (u''/s)=0.457\ldots <\phi (d'/t')+\phi (t''/{d'}')=0.467\ldots ~,~~~\,
          {\rm so}$$
 $$[\phi (u''/s)|\leq [\phi (d'/t')]+[\phi (t''/d'')]+2\leq m_3+5$$
            (see the equalities and inequalities (\ref{eqphi}) and
            (\ref{eq[]})) and one
            concludes that $\ell _3\leq m_3+6$. The factors $|K(q^m)|$
            which have not been
            mentioned up to now are all of modulus $<1$; the corresponding 
            numbers $q^m$ belong to the intervals $(0,d')$ and $(t'',u'')$. Thus
            $\prod _{m=1}^{\infty}|K(q^m)|<2^6$. At the same time

            $$\prod_{m=1}^{\infty}(1-q^m)\leq \prod_{m=1}^{\infty}(1-0.8^m)<
            7\times 10^{-6}~.$$
            This shows that $|\Theta ^*(q,-5)|<10^{-4}<4/25<|-G(q,-5)|$
            from which
            the proposition follows.
            
            %$$\ln (s/u'')=0.4575475384\ldots <(\ln t'/d')+\ln (d''/t'')=
           % 0.4676587897\ldots ~,$$

\section{Some technical lemmas\protect\label{secprelim}}

In this section we formulate and prove several lemmas:

\begin{lm}\label{lmG}
  For $q\in [0,1]$ and $|x|=a>2$, one has $|G|\geq (a-2)/a(a-1)$.
  In particular, for $a=3$, $|G|\geq 1/6$. 
\end{lm}

\begin{proof}
  Indeed,

  $$\begin{array}{cclcl}
    |G|&\geq&(1/|x|)(1-\sum _{j=-\infty}^{-2}q^{j(j+1)/2}|x^{j+1}|)&\geq&
    (1/|x|)(1-\sum _{j=-\infty}^{-2}|x^{j+1}|)\\ \\
    &=&(1/a)(1-\sum _{j=-\infty}^{-2}a^{j+1})&=&(a-2)/a(a-1)~.\end{array}$$
\end{proof}

We set $X:=1/x$ and we represent the function $G$ in the form
$G=G_5+G_*$, $G_5:=X+qX^2+q^3X^3+q^6X^4+q^{10}X^5$,
$G_*:=\sum _{j=5}^{\infty}q^{j(j+1)/2}X^{j+1}$. 

\begin{lm}\label{lm147}
  For $(q,t)\in [0.6,1]\times [0,w]$, one has $|G_*|<0.0208$ and
  $|G_5|>0.147$.
\end{lm}

\begin{proof}
  For $(q,t)\in [0.6,1]\times [0,w]$, it is true that
  $|X|=1/|-t+wi|\leq 1/w$ and 

  $$|G_*|\leq \sum _{j=5}^{\infty}|q^{j(j+1)/2}X^{j+1}|\leq
  \sum _{j=5}^{\infty}|X^{j+1}|\leq \sum _{j=5}^{\infty}|w^{-j-1}|=
  0.02076055760\ldots <0.0208$$
  which proves the first claim of the lemma. To prove the second one we
  represent the function $G_I:=$Im$(G_5(q,1/(-t+wi)))$ in the form

  $$G_I=(3\sqrt{2}/(2t^2+9)^5)G^{\flat}_I~,~~~\, {\rm where}~~~\,
  G^{\flat}_I=g_0+g_1q+g_3q^3+g_6q^6+g_{10}q^{10}~,~~~\, {\rm with}$$

  $$\begin{array}{cclccl}
    g_0&:=&-16t^8-288t^6-1944t^4-5832t^2-6561~,&g_6&:=&64t^5-1296t~,\\ \\
    g_1&:=&32t^7+432t^5+1944t^3+2916t~,&&\\ \\
    g_3&:=&-48t^6-360t^4-324t^2+1458~,&g_{10}&:=&-80t^4+720t^2-324~.\end{array}$$
  Our first goal is to give an upper bound for $G_I$ for $t\in [0,1]$
  (hence an upper bound for $G_I^{\flat}$). We use the evident equalities and
  inequalities

  $$\begin{array}{llll}
    288=256+32~,&1944=1512+432~,&5832=3168+1944+720~,&
    6561=2916+1458+2187~,\\ \\ 
    32t^6\geq 32t^7q~,&432t^4\geq 432t^5q~,&1944t^2\geq 1944t^3q~,&
    2916\geq 2916tq~,\\ \\
    64tq^6\geq 64t^5q^6~,&720t^2\geq 720t^2q^{10}~,&1296=64+1232&1458\geq 1458q^3~
  \end{array}$$
  to obtain an upper bound $G^u$ for $G_I^{\flat}$ in which all coefficients are
  negative:

  $$\begin{array}{ccl}G^u&:=&-16t^8-256t^6-1512t^4-3168t^2-2187\\ \\
    &&-(48t^6+360t^4+324t^2)q^3-1232tq^6-(80t^4+324)q^{10}~.\end{array}$$
  For $t\in [0,1]$ fixed, the upper bound of the product
  $(3\sqrt{2}/(2t^2+9)^5)G^u$ is attained for $q=0.6$. The function
  $(3\sqrt{2}/(2t^2+9)^5)G^u|_{q=0.6}$ is decreasing in $t$ and its value
  for $t=0$ is $v_1:=-0.1572756008\ldots$, so $v_1$ is the upper bound of $G_I$
  for $t\in [0,1]$, $q\in [0.6,1]$.

  Suppose now that $t\in [1,w]$ and $q\in [0.6,1]$. We observe first that
  the function $G_I|_{q=1}$ is increasing and
  $(G_I|_{q=1})|_{t=w}=-0.1478254790\ldots =:v_2$. Next, we prove that

   $$\partial G^{\flat}_I/\partial q=g_1+3qg_3+6q^5g_6+10q^9g_{10}>0~~~\,
  {\rm hence}~~~\, \partial G_I/\partial q>0~.$$
  The quantities $g_1$, $g_6$ and $g_{10}$ do not change sign for $t\in [1,w]$:
  $g_1>0$, $g_6\leq 0$, with equality only for $t=w$, while $g_{10}>0$.
  The polynomial $g_3$ is negative for $t>t_1:=1.224744871\ldots$ and
  positive for $t\in [1,t_1)$; it vanishes for $t=t_1$. Thus for $t\geq t_1$,

    $$\partial G^{\flat}_I/\partial q\geq g_1+3g_3+6g_6+10\times 0.6^{10}g_{10}=:
    G^{\ddagger}~.$$
    The polynomial $(G^{\ddagger})'$ has a single real root
    $t_2:=1.144295977\ldots$. The function $G^{\ddagger}$ is increasing for
    $t\geq t_2$, so
    for $t\geq t_1>t_2$, one has
    $G^{\ddagger}(t)>G^{\ddagger}(t_2)=9.468005\ldots >0$.

     For $t\in [1,t_1]$, we minorize the function
    $\partial G^{\flat}_I/\partial q$ in each of the four cases
    $q\in [0.6,0.7]$, $q\in [0.7,0.8]$, $q\in [0.8,0.9]$ and $q\in [0.9,1]$.
    Denote any of these four intervals by $[a,b]$. The minoration is
    looked for in the form

    $$G_{a,b}:=g_1+3a^2g_3+6b^5g_6+10a^9g_{10}~.$$
    Each of the four functions $G_{a,b}$ turns out to be monotone increasing
    on $[1,t_1]$, so the four minima are attained for $t=1$. They equal
    $4897.5\ldots$, $4096.5\ldots$, $2777.1\ldots$ and $920.4\ldots$
    respectively. Hence for $t\in [1,w]$, $\partial G^{\flat}_I/\partial q>0$
    and the function $G_I$ is maximal for $q=1$. Hence it is
    $\leq v_2$. For $t\in [0,1]$, it is $\leq v_1$, so it is $\leq v_2<-0.147$
    for $(q,t)\in [0.6,1]\times [0,w]$ and $|G_I|>0.147$.

\end{proof}

\begin{lm}\label{lmcos}
  Consider the factors $|1+q^mx|$ and $|1+q^{m-1}/x|$, $m=1$, $2$, $\ldots$.

  (1) For $q\in (0,1)$ fixed and for $x\in C_1\cup C_2$,
  these quantities are decreasing functions in $\varphi :={\rm arg}x$.

  (2) For $x=3e^{3\pi i/4}$, each factor $|1+q^{m-1}/x|$, $m\geq 2$,
  is a decreasing function
  in $q\in (0,1)$.

  (3) For $q\in [0.5,1]$ and for $x=3e^{3\pi i/4}$, the factor $|1+qx|$ is an
  increasing function in~$q$.
\end{lm}

\begin{proof}
  The first claim of the lemma follows from the cosine theorem. Indeed,
  recall that $1/x=\bar{x}/|x|^2$, so arg$(1/x)=-$arg$x=-\varphi$ and
  $\cos ($arg$(1/x))=\cos \varphi$. Hence 

  $$\begin{array}{ccl}
    |q^mx-(-1)|^2&=&q^{2m}|x|^2+(-1)^2-2(-1)q^m|x|\cos \varphi\\ \\ &=&
    q^{2m}|x|^2+1+2q^m|x|\cos \varphi~,\\ \\
    |(q^{m-1}/x)-(-1)|^2&=&q^{2m-2}/|x|^2+(-1)^2-2(-1)(q^{m-1}/|x|)
    \cos \varphi\\ \\ &=&
    q^{2m-2}/|x|^2+1+2(q^{m-1}/|x|)\cos \varphi~.
  \end{array}$$
  For $q$ fixed, these quantities are decreasing in $\varphi$, because
  such is $\cos \varphi$. Set $\cos \varphi :=-\sqrt{2}/2$, $|x|:=3$.
  The displayed formulas show that

  $$\begin{array}{llll}
    d(|1+q^{m-1}/x|)/dq&=&((m-1)q^{m-2}/|x|^2)(2q^{m-1}-\sqrt{2}|x|)<0&
    {\rm and}\\ \\
    d(|1+qx|)/dq&=&m|x|(2q|x|-\sqrt{2})>0&\end{array}$$
  from which one deduces the last two claims of the lemma.
  \end{proof}

In the proofs we need some properties of the functions $M:=|(1+qx)(1+q/x)|$
and $M_0:=(1-q)M$. We remind that we set $x=-t+wi$, $t\in [0,w]$, $w=3/\sqrt{2}$.

\begin{lm}\label{lmmaxmodulus}
      For $t\in [0,w]$ and for $q\in [0.6,1]$ fixed, the quantities $M$ and $M_0$ 
      are maximal
      for $t=0$.  For $q\in [0.6,0.75]$ fixed and for $t\in [1,w]$, they are maximal for $t=1$.
    \end{lm}

\begin{proof}
      It suffices to prove the claims of the lemma about the function $M$. One checks directly for the square of $M$ that

      $$M^2=(2q^2t^2+9q^2-4qt+2)(2q^2-4qt+2t^2+9)/(2t^2+9)~.$$
      One verifies straightforwardly that 
      
      $$\begin{array}{l}
         M^2-M^2|_{t=0}=-2qtP/(2t^2+9)~,~~~\, {\rm where}\\ \\ 
         P:=36q^2t^2-18qt^3+ 198q^2-149qt+36t^2+ 198~.
        \end{array}$$
The discriminant of the trinomial  $36(qt)^2-149qt+198$ is negative, so this trinomial is positive-valued. For the remaining terms of $P$, for $t\in [0,w]$ (hence $t^2\leq 9/2$), one obtains

$$-18qt^3+198q^2+36t^2\geq -81qt+198q^2+36t^2$$
which is again a trinomial with negative discriminant. Thus $P>0$ and $M^2-M^2|_{t=0}\leq 0$ with equality only for $t=0$ which proves the first claim of the lemma. To prove its second statement we consider the difference 

$$\begin{array}{l}M^2-M^2|_{t=1}=-2q(t-1)(V_2q^2+V_1q+V_0)/11(2t^2+9)~,~~~\, {\rm where}\\ \\ 
   V_2=V_0:=44t^2-8t+234~~~\, {\rm and}~~~\, V_1:=-(t+1)(22t^2+167)~.
  \end{array}
$$
The polynomial $V_2q^2+V_1q+V_0$ has no crfitical points for $(q,t)\in [0.6,0.75]\times [1,w]$. Its restrictions to each of the sides of this rectangle (i.~e. its restrictions obtained for $q=0.6$, $q=0.75$, $t=1$ and $t=w$) are positive-valued. Hence the difference $M^2-M^2|_{t=1}$ is negative in the given rectangle which proves the second claim of the lemma.

\end{proof}

 \begin{rem}\label{remM}
{\rm For $x=-t+wi$, we represent in Fig.~\ref{twographs} the graph of the function}

$$\begin{array}{ccl}
   M_1(q,t)&:=&M_0(q,-t+wi):=(1-q)|(1+qx)(1+q/x)|\\ \\
   &=&(1-q)(2q^2t^2+9q^2-4tq+2)^{1/2}
  (2q^2-4tq+2t^2+9)^{1/2}/(2(2t^2+9))^{1/2}\end{array}$$
 {\rm for two fixed values of $t$, namely
    $t=0$ (in solid line) and
   $t=1$ (in dashed line). The two functions}

\begin{figure}[htbp]
%\includegraphics[scale=0.5]{parthetanegfirstfour.eps}
%\centerline{\hbox{\includegraphics[scale=0.7]{parthetanegfirstfour.eps}}}
\centerline{\hbox{\includegraphics[scale=0.7]{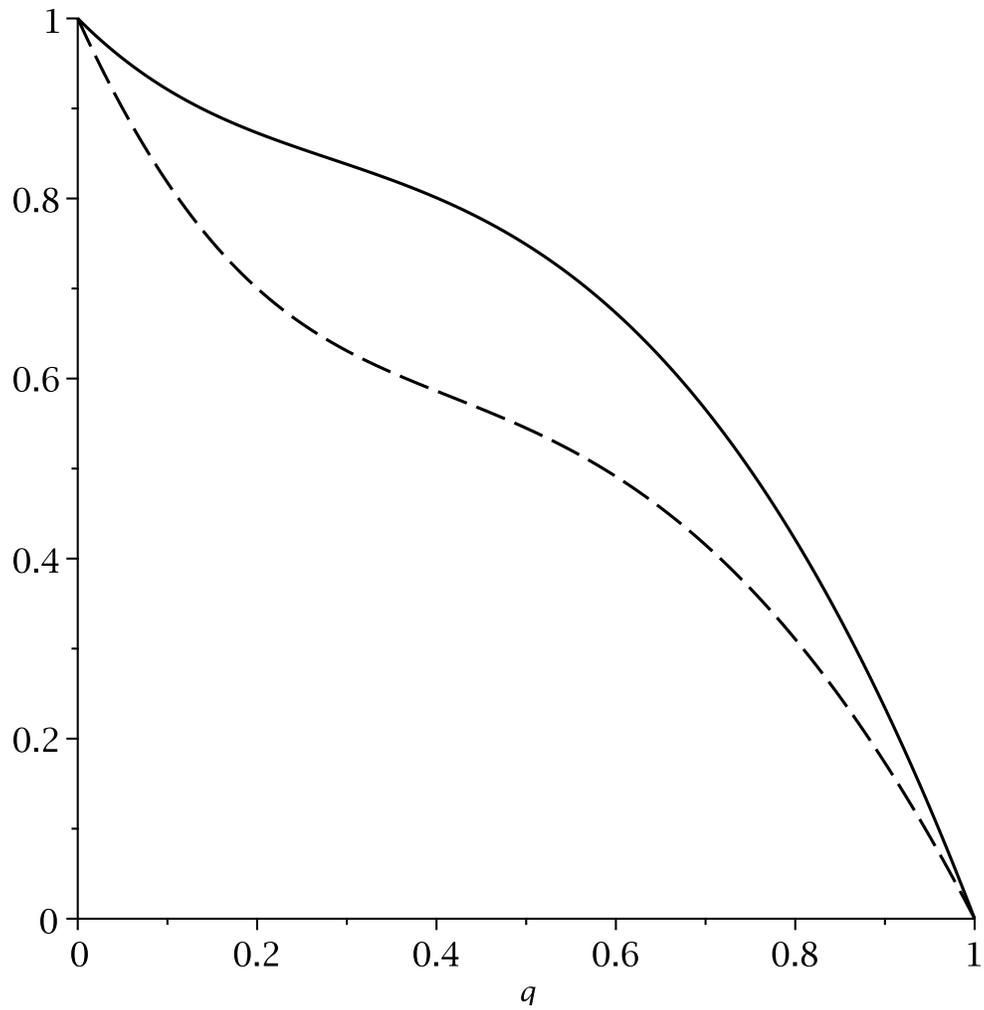}}}
%\centerline{\hbox{\epsfxsize=10cm \epsfbox{k=1234.pdf}}}
\caption{The graphs of the functions $M_1(q,0)$ (in solid line) and
  $M_1(q,1)$ (in dashed line).}
\label{twographs}
          \end{figure}
%{\rm The two functions} 

$$\begin{array}{l}
M_1(q,0)=(1-q)(9q^2+2)^{1/2}
  (2q^2+9)^{1/2}/3\sqrt{2}~~~\, {\rm and}\\ \\ M_1(q,1)=(1-q)(11q^2-4q+2)^{1/2}
  (2q^2-4q+11)^{1/2}/\sqrt{22}\end{array}$$
  {\rm are decreasing on $[0,1]$.}
%{\rm For $t=1$, one finds that if $q\in [0.75,1]$, then} 

%$$|\Theta ^*|\leq |3i/\sqrt{2}|\prod _{m=1}^{\infty}M_1(0.75^m,1)\leq
%(3/\sqrt{2})\prod _{m=1}^{40}M_1(0.75^m,1)=0.2341273795\ldots ~.$$

%{\rm For $t=0$, a direct computation yields}

%$$|\Theta ^*|=|1+3i/\sqrt{2}|\prod _{m=1}^{\infty}M_1(q^m,0)\leq
%\sqrt{11/2}\prod _{m=1}^{40}M_1(q^m,0)=0.2588375569\ldots ~.$$
\end{rem}

\section{Plan of the proof of
  Theorem~\protect\ref{tmmain}\protect\label{secplan}}

The zeros of $\theta$ depend continuously on $q$ and no zeros are born at
$\infty$. We prove that for $q\in (0,1)$, there is no zero of $\theta$ on the
border $\partial \mathcal{D}$ of the domain $\mathcal{D}$. For $q\in (0,0.5]$,
          this follows from the proposition below. We remind that (see Notation~\ref{notat})

$$\partial \mathcal{D}=C_2\cup C_3\cup S_+\cup S_-\cup S_v~.$$

          \begin{prop}\label{prop05}
  For $q\in (0,0.5]$, the function $\theta (q,.)$ has no zeros in the
  closed rectangle
  $\Delta :=\{ -3\leq {\rm Re}x\leq 0,~-3\leq {\rm Im}x\leq 3\}$.
\end{prop}

The proof of this proposition and of all lemmas formulated in this
section are given in Section~\ref{secproofs}. The rectangle
$\Delta$ contains the domain $\mathcal{D}$, so for $q\in (0,0.5]$, there
are no zeros of $\theta$ on $\partial \mathcal{D}$. One can observe that
for $q\in (0,\tilde{q}_1]$, $\tilde{q}_1=0.3092\ldots$,
          there are no complex conjugate pairs of $\theta$ (see Section~\ref{secremind}),
and for $q\in (\tilde{q}_1,0.5]$, there is exactly one such pair.  
            
  From now on we assume that $q\in [0.5,1)$. The next lemma explains why
    no zeros of $\theta$ can be found on the arc~$C_2$ hence none on the arc
    $C_3$ either.

    \begin{lm}\label{lm051}
  For $q\in [0.5,1)$ and $x\in C_2$, one has $|G|>|\Theta ^*|$
    hence $|\theta |>0$.
  \end{lm}

    The next lemma states that the function $\theta$ has no zeros on the
    segment $S_v$:

\begin{lm}\label{lmnoimaginary}
  For $q\in (0,1)$, the function $\theta (q,.)$ has no purely imaginary zeros
  of modulus $\leq 2.2$ hence no such zeros of modulus
  $\leq 3/\sqrt{2}=2.1\ldots$.
\end{lm}

It remains to show that there are no zeros of $\theta$ on the segments
$S_{\pm}$. It suffices to deal with the segment $S_+$. We consider the
restrictions to $[0.3,0.6]\times [0,w]$
of the functions $\theta _5(q,t):=\sum _{j=0}^5q^{j(j+1)/2}(-t+wi)^j$ and
$\theta _*(q,t)=\sum _{j=6}^{\infty}q^{j(j+1)/2}(-t+wi)^j$.

\begin{lm}\label{lm0306}
  For $(q,t)\in [0.3,0.6]\times [0,w]$, one has $|\theta _*(q,t)|\leq 0.018$
  and $\theta _I:=${\rm Im}$(\theta _5(q,t))>0.13$. Hence for
  $(q,t)\in [0.3,0.6]\times [0,w]$, the
  function $\theta$ has no zeros.
\end{lm}

Next we settle the case $q\in [0.75,1)$.

   \begin{lm}\label{lm0751}
For $(q,t)\in [0.75,1)\times [0,w]$, the function $\theta$ has no zeros.
      \end{lm}

   The remaining case $q\in [0.6,0.75]$ will be subdivided in two cases:

      \begin{lm}\label{lm1w}
For $(q,t)\in [0.6,0.75]\times [1,w]$, the function $\theta$ has no zeros.
      \end{lm}

\begin{lm}\label{lm01}
  For $(q,t)\in [0.6,0.75]\times [0,1]$, the function $\theta$ has no zeros.
  \end{lm}

\section{Proofs\protect\label{secproofs}}

\begin{proof}[Proof of Proposition~\ref{prop05}]
  A) For $q\in [0,0.3]$, all zeros of $\theta (q,.)$ are real, negative and
  distinct, see Section~\ref{seccomments}. All these zeros are $<-5<-3$,
  see Proposition~\ref{propmain}.
  %For $x\in [-1/q,0)$, the series $\theta (q,x)$
  %    is a Leibniz series with positive leading term, so $\theta (q,x)>0$ and
  %    the rightmost zero of $\theta (q,.)$ is $<-1/q\leq -1/0.3<-3$.
      %So we suppose that $q\in (0.3,0.5]$.
\vspace{1mm}

B) We set $\theta =\theta _4^{\bullet}+\theta _*^{\bullet}$, where
$\theta _4^{\bullet}:=1+qx+q^3x^2+q^6x^4+q^{10}x^4$ and
$\theta _*^{\bullet}:=\sum _{j=5}^{\infty}q^{j(j+1)/2}x^j$. For $x\in \Delta$, one has
$|x|\leq 3\sqrt{2}=4.24\ldots <4.25$, so for $(q,x)\in [0,0.5]\times \Delta$,
one obtains the majoration

\begin{equation}\label{eqestim}
  |\theta _*^{\bullet}(q,x)|\leq
  \sum _{j=5}^{\infty}0.5^{j(j+1)/2}4.25^j=0.045\ldots <0.046~.
\end{equation}

C) We denote the border of the rectangle $\Delta$ by $\partial \Delta$
and we set $I_0:=[0,0.5]$. We
show that for each $q\in I_0$ fixed, one has
\begin{equation}\label{eqRouche}
  |\theta _4^{\bullet}(q,x)|>|\theta _*^{\bullet}(q,x)|>0
\end{equation}
for any $x\in \partial \Delta$. For
$q\in [0,0.01]$, there is no zero of $\theta _4^{\bullet}$ in
  $\Delta$. Indeed, one
  obtains 

  $$|\theta _4^{\bullet}|\geq 1-0.01\times 4.25-0.01^3\times 4.25^2-0.01^6\times
  4.25^3-0.01^{10}\times 4.25^4=0.95\ldots >0~.$$
  The condition (\ref{eqRouche}) is fulfilled for
  $x\in \partial \Delta$, so it implies that no zero of $\theta _4^{\bullet}$
  may enter $\Delta$
  as $q$ increases from $0.01$ to $0.5$. Hence $\theta _4^{\bullet}$
  has no zeros in $\Delta$ for $q\in I_0$.

  Again from condition (\ref{eqRouche}) and from the Rouch\'e theorem
  follows that $\theta$ has no zeros for $(q,x)\in I_0\times \Delta$. So
  our aim is to show that  condition~(\ref{eqRouche}) holds true.
  \vspace{1mm}
  
  D) When proving condition (\ref{eqRouche}) we deal only with the part of
  $\partial \Delta$ with Im$x\geq 0$. For Re$x=-3$, we set $x:=-3+it$,
  $t\in [0,3]$. Then

  $$\begin{array}{ccccl}G_R(q,t)&:=&{\rm Re}(\theta _4^{\bullet})&=&
    q^{10}t^4-54q^{10}t^2+81q^{10}+9q^6t^2-27q^6-q^3t^2+9q^3-3q+1~,\\ \\
    G_I(q,t)&:=&{\rm Im}(\theta _4^{\bullet})&=&
    12q^{10}t^3-108q^{10}t-q^6t^3+27q^6t-6q^3t+qt~.\end{array}$$
  We use the fact that
  $|\theta _4^{\bullet}|\geq \max (|{\rm Re}(\theta _4^{\bullet})|,
  |{\rm Im}(\theta _4^{\bullet})|)=:\mu$.
  Neither of the functions $G_R$ and $G_I$ has a critical point with $q\in I_0$,
  so $G_R$ (resp. $G_I$) attains its maximal and its
  minimal value when one
  of the following conditions takes place: $t=0$, $t=3$, $q=0$ or $q=0.5$.
  
  For $q=0$, one has $G_R\equiv 1$, so $\mu \geq 1>0.046$. For $q=0.5$,
  one gets

  $$G_R=0.0009765625t^4-0.037109375t^2+0.2822265625~,~~~\,
  G_I=-0.00390625t^3+0.06640625t$$
  and one checks directly that for $t\in [0,1]$ and $t\in [1,3]$, one has
  $G_R>0.05>0.046$ and $G_I>0.05>0.046$ respectively.

  For
  $t=0$, one obtains $G_R=81q^{10}-27q^6+9q^3-3q+1$ which is $>0.2>0.046$
  for $q\in I_0$. For $t=3$, it is clear that 

  $$G_R=-324q^{10}+54q^6-3q+1~~~\, {\rm and}~~~\,
  G_I=54q^6-18q^3+3q~,$$
  with $G_R>0.2>0.046$ for $q\in [0,0.2]$ and with
  $G_I>0.05>0.046$ for $q\in [0.2,0.5]$ respectively.
  \vspace{1mm}

  E) For Re$x=0$, one sets $x:=i\tau$ to obtain

  $$U_R(q,\tau ):={\rm Re}(\theta _4^{\bullet})=
  q^{10}\tau ^4-q^3\tau ^2+1~~~\, {\rm and}~~~\,
  U_I(q,\tau ):={\rm Im}(\theta _4^{\bullet})=-q\tau (q^5\tau ^2-1)~.$$
  Neither of the functions $U_R$ and $U_I$ has a critical point with $q\in I_0$,
  so their maximal and minimal values are attained for $\tau =0$, $\tau =3$,
  $q=0$ or $q=0.5$. In each of the cases $\tau =0$ and  $q=0$ one has $U_R\equiv 1>0.046$.

  Suppose that $\tau =3$. Then $U_R>0.05>0.046$ for $q\in [0,0.3]$ and
  $U_I>0.05>0.046$ for $q\in [0.3,0.5]$. Finally, if $q=0.5$, then

  $$U_R=0.0009765625\tau ^4-0.125\tau ^2+1~~~\, {\rm and}~~~\,
  U_I=-0.5\tau (0.03125\tau ^2-1)~,$$
  with $U_R>0.05$ for $\tau \in [0,2]$ and with $U_I>0.05$ for
  $\tau \in [2,3]$.
  \vspace{1mm}

  F) Suppose that Im$x=3$. Then we set $x:=u+3i$, $u\in [-3,0]$. Then

  $$\begin{array}{ccccl}
    S_R(q,u)&:=&{\rm Re}(\theta _4^{\bullet})&=
    &q^{10}u^4-54q^{10}u^2+81q^{10}+q^6u^3-27q^6u+
    q^3u^2-9q^3+qu+1~,\\ \\
    S_I(q,u)&:=&{\rm Im}(\theta _4^{\bullet})&=
    &12q^{10}u^3-108q^{10}u+9q^6u^2-27q^6+6q^3u+3q~.
  \end{array}$$
  The functions $S_R$ and $S_I$ have no critical points for $(q,u)$ inside
  the rectangle $I_0\times [-3,0]$, so their maximal and minimal values are
  attained on its border. Obviously $S_R|_{q=0}\equiv 1>0.046$, $S_R|_{u=0}=81q^{10}+1\geq 1>0.046$ and

  $$S_I|_{q=0.5}=0.01171875u^3+0.140625u^2+0.64453125u+1.078125$$
  which is $>0.05>0.046$ for $u\in [-3,0]$. For $u=-3$, one obtains

  $$S_R=-324q^{10}+54q^6-3q+1~~~\, {\rm and}~~~\, S_I=54q^6-18q^3+3q~,$$
  with $S_R>0.05>0.046$ for $q\in [0,0.2]$ and with $S_I>0.05>0.046$
  for $q\in [0.2,0.5]$.

\end{proof}

\begin{proof}[Proof of Lemma~\ref{lm051}]
  It suffices to show that $|\Theta ^*|<1/6$, see Lemma~\ref{lmG} with $a=3$.
  By part (1) of Lemma~\ref{lmcos}, it is
  sufficient to prove this for $x=\lambda :=3e^{3\pi i/4}$. The modulus 
  $|R|:=\prod_{m=1}^{\infty}|1+q^{m-1}/x|$ is maximal (see Notation~\ref{notaxXt} and part (2) of  Lemma~\ref{lmcos}) when $q=0.5$ in which case one gets 

  $$|1+q^{m-1}/x|=r_m:=|1+0.5^{m-1}(-\sqrt{2}-\sqrt{2}i)/6|=
  ((1-0.5^{m-1}\sqrt{2}/6)^2+0.5^{2m-2}/18)^{0.5}$$
  and one finds numerically that

  $$|R|\leq \prod _{m=1}^{11}r_m=0.6329437509\ldots <0.633~.$$

  Next, for $x=\lambda$, the points representing the complex numbers
  $u_m:=1+xq^m$ lie on the
  straight line passing through the points $1$ and $i$; they lie above the
  abscissa-axis. We denote by $m_0\in \mathbb{N}$ the index $m$ for which
  Re$(u_m)\leq 0$ (hence $|u_m|\geq 1$) and Re$(u_{m+1})>0$ (hence $|u_{m+1}|<1$).
  One has $m_0\geq 1$. Indeed, for $q\in [0.5,1]$, 

  $$|1+q\lambda |\geq |1+0.5\lambda |=1.062393362\ldots >1$$
  (see part (3) of Lemma~\ref{lmcos}). A numerical check shows that one
  has $m_0\geq 2$ exactly if $q\geq 0.6865890479\ldots >0.68=:q^{\dagger}$.
  
  For $m\leq m_0$, one has $|u_m|\leq |v_m|$, where
  $v_m:=q^m+\lambda q^m=q^m(1+\lambda )$ (with equality only for $q=1$). One finds numerically that

  $$|1+\lambda |=2.399449794\ldots <2.4~,$$
  so for $m\leq m_0$, $|u_m|<2.4$. For each product $p_m:=(1-q^m)u_m$ one can write

$$|p_m|\leq |(1-q^m)v_m|=|(1-q^m)q^m||1+\lambda |\leq (1/4)\times 2.4=0.6~.$$
  Suppose that $m_0\geq 3$. Then

  $$\begin{array}{ccl}
    |\Theta ^*|&\leq&(\prod _{m=1}^{m_0}|p_m|)\times 0.633\times
    (\prod _{m=m_0+1}^{\infty}|u_m|)\times \prod _{m=m_0+1}^{\infty}(1-q^m)\\ \\
    &\leq&0.6^{m_0}\times 0.633\leq 0.6^3\times 0.633=0.136728<1/6~.\end{array}$$
  Suppose that $m_0=2$. The maximal value of the function
  $q(1-q)q^2(1-q^2)=q^3(1-q)^2(1+q)$ for $q\in [0,1]$ is
  $0.05579835315\ldots <0.056$; it is attained for $q=0.6286669788\ldots$.
  Thus

  $$|p_1||p_2|\leq 0.056\times 2.4^2~,~~~\, \prod _{m=3}^{\infty}(1-q^m)<0.78~~~\,
  {\rm and}$$ 

  $$\begin{array}{ccl}
    |\Theta ^*|&\leq&|p_1||p_2|\times 0.633\times
    (\prod _{m=3}^{\infty}|u_m|)\times \prod _{m=3}
    ^{\infty}(1-q^m)\\ \\
    &\leq&0.056\times 2.4^2\times 0.633\times 0.78<0.16<1/6~.\end{array}$$
  Suppose that $m_0=1$. Then $0.5\leq q<0.69$. One finds that 
  
  $$\begin{array}{l}
   \prod_{m=2}^{\infty}(1-q^m)<\prod_{m=2}^{100}(1-0.5^m)=0.5775\ldots <0.5776~~~\, {\rm and}\\ \\ 
   \prod_{m=2}^{\infty}|u_m|<\prod_{m=2}^{30}|u_m|=:g(q)~.
    \end{array}$$
The function $g$ is decreasing for $q\in [0.5,0.69]$ and $g(0.5)=0.4254\ldots <0.4255$. Therefore

$$\begin{array}{ccl}|\Theta ^*|&\leq&|p_1|\times 0.633\times g(q)\times 0.5776\\ \\ 
   &<&0.6\times 0.633\times 0.4255\times 0.5776=0.093\ldots <1/6~.
  \end{array}$$

  \end{proof}

\begin{proof}[Proof of Lemma~\ref{lmnoimaginary}]
  Indeed, set $x:=iy$, $y\in \mathbb{R}$. Hence
  
  $$\theta (q,iy)=\theta (q^4,-y^2/q)+iqy\theta (q^4,-qy^2)~.$$The first and the
  second summand represent the real and the imaginary part of $\theta$ when
  restricted to the imaginary axis. If $iy_0$ is a zero of $\theta (q,.)$,
  $y_0\in \mathbb{R}$, 
  then

  $$\theta (q^4,-y_0^2/q)=\theta (q^4,-qy_0^2)=0~.$$
  By Proposition~\ref{propmain}, $-y_0^2/q<-5$ and
  $-qy_0^2<-5$, so $|y_0|>\sqrt{5}>2.2$.
  \end{proof}

\begin{proof}[Proof of Lemma~\ref{lm0306}]
  For $t\in [0,w]$, one has $|-t+wi|\leq 3$, with equality only for $t=w$.
  Thus

  $$|\theta _*(q,t)|\leq \sum _{j=6}^{\infty}q^{j(j+1)/2}3^j\leq
  \sum _{j=6}^{\infty}0.6^{j(j+1)/2}3^j=0.017\ldots <0.018~.$$
  One finds by direct computation that

  $$\theta _I=(3\sqrt{2}q/8)(20q^{14}t^4-180q^{14}t^2+81q^{14}-16q^{9}t^3+
  72q^{9}t+12q^5t^2-18q^5-8q^2t+4)~.$$
  By lowercase indices $t$ or $q$ we denote derivations w.r.t. these variables.
  We show first that $(\theta _I)_t<0$. Thus for $q\in [0.3,0.6]$ fixed,
  $\theta _I$ is maximal for $t=0$ and minimal for $t=w$. One finds that

$$\begin{array}{cclccl}
        (\theta _I)_t&=&3\sqrt{2}(10q^{12}t^3-45q^{12}t-6q^7t^2+
        9q^7+3q^3t-1)q^3~,&&&\\ \\ 
        (\theta _I)_{tt}&=&9\sqrt{2}q^6(10q^9t^2-15q^9-4q^4t+1)~,&
        (\theta _I)_{ttt}&=&36\sqrt{2}q^{10}(5q^5t-1)~.\end{array}$$
  For $(q,t)\in [0.3,0.6]\times [0,w]$, one has $(\theta _I)_{ttt}<0$. Hence
  $(\theta _I)_{tt}$ is minimal for $t=w$; in this case it equals

  $$9\sqrt{2}q^6(30q^9-6q^4\sqrt{2}+1)$$
  which is positive for $q>0$. Therefore
  $(\theta _I)_t$ is maximal for $t=w$ when it equals

$$3\sqrt{2}(-18q^7+(9\sqrt{2}q^3)/2-1)q^3$$
  which is negative for $q>0$.
  So $\theta _I$ is minimal for $t=w$ and

  $$\theta _I(q,w)=-(3\sqrt{2}q(81q^{14}-9q^5+3\sqrt{2}q^2-1))/2~.$$
  The derivative $(\theta _I(q,w))_q$ is negative for $q\geq 0.3$,
  this means that $\theta _I$ is minimal for $(q,t)=(0.6,w)$. One has
  $\theta _I(0.6,w)=0.1387526518\ldots >0.018$, so for $q\in (0,0.6]$,
  $\theta$ has no zeros for $t\in [0,w]$. 
  \end{proof}

 \begin{proof}[Proof of Lemma~\ref{lm0751}]
      For $q\in [0.75,1)$ fixed and for $t\in [0,w]$, the quantity 
      
      \begin{equation}\label{eqM1}
      M_1:=(1-q)|(1+qx)(1+q/x)||_{x=-t+wi}\end{equation}
      %~~~\, {\rm and}~~~\,   M_0:=(1-q)M$$
      is maximal for
        $t=0$, see Lemma~\ref{lmmaxmodulus}. The quantity $M_1|_{t=0}$ is maximal for $q=0.75$, see Remark~\ref{remM}.  In this case 

        $$\begin{array}{ccl}
          \prod_{m=1}^{\infty}M_1(q^m,t)&\leq&
          \prod_{m=1}^{\infty}M_1(0.75^m,0)\\ \\ 
          &<&\prod_{m=1}^{40}M_1(0.75^m,0)=0.1103687051\ldots =:h_1~,\\ \\
          |1+1/x|&=&|1+1/(-t+wi)|=|(1-t)+wi|/|-t+wi|\\ \\
          &=&f(t):=(((1-t)^2+w^2)/(t^2+w^2))^{1/2}\\ \\ &\leq&
          (11/9)^{1/2}=1.105541597\ldots =:h_2
          \end{array}$$
        and $|\Theta ^*|\leq h_1h_2=0.1220171945\ldots <0.123$, see (\ref{eqthetaG}). (It is easy
        to show that the function $f$ is decreasing on $[0,w]$, so
        $f(t)\leq f(0)=(11/9)^{1/2}$.) 
        On the other hand Lemma~\ref{lm147} implies 

$$|\theta |\geq |G|-|\Theta ^*|\geq |G_5|-|G_*|-|\Theta ^*|\geq
        0.147-0.0208-0.123=0.0032>0~.$$

 \end{proof}

\begin{proof}[Proof of Lemma~\ref{lm1w}]
 For $q\in [0.6,0.75]$ fixed and for $t\in [1,w]$, the quantity $M_1$ (see (\ref{eqM1})) 
 %$M:=|(1+qx)(1+q/x)||_{x=-t+wi}$ 
 is maximal for
        $t=1$ (see Lemma~\ref{lmmaxmodulus}). The quantity $M_1|_{t=1}$ is maximal for $q=0.6$, see Remark~\ref{remM}. Therefore it is true that

        $$
        %\begin{array}{ccl}
          \prod_{m=1}^{\infty}M_1(q^m,t)\leq
          \prod_{m=1}^{\infty}M_1(0.6^m,1)<\prod_{m=1}^{40}M_1(0.6^m,1)=0.1048026086\ldots =:h_3~.
          %\end{array}
          $$
        The function $f$ defined in the proof of Lemma~\ref{lm0751} is decreasing and takes its maximal value $h_4:=(9/11)^{1/2}$ for $t=1$. Thus
        $|\Theta ^*|\leq h_3h_4=
        0.09479752467\ldots <0.095$. Using
        Lemma~\ref{lm147}, one deduces that

         $$|\theta|\geq |G_5|-|G_*|-|\theta ^*|\geq 0.147-0.0208-0.095=0.0312>0~.$$
      \end{proof}

\begin{proof}[Proof of Lemma~\ref{lm01}]
  We set $\theta (q,x)=\theta _7^{\bullet}(q,x)+\theta _{*}^{\bullet}(q,x)$, where
  $\theta _7^{\bullet}:=\sum _{j=0}^7q^{j(j+1)/2}x^j$ and
  $\theta _{*}^{\bullet}:=\sum _{j=8}^{\infty}q^{j(j+1)/2}x^j$. The maximal possible
  modulus $|x|$ for $t\in [1,w]$ is obtained for $t=1$. This gives $x= -1+3i/\sqrt{2}$ and  in this case
  $|x|=2.345207880\ldots <2.346$. This means that

  $$|\theta _{*}^{\bullet}|\leq \sum _{j=8}^{\infty}|q|^{j(j+1)/2}|x|^j\leq
    \sum _{j=8}^{\infty}0.75^{j(j+1)/2}2.346^j<0.036~.$$
    On the other hand when setting $x:=-t+iw$, $t\in [0,1]$, and
    $T_I:={\rm Im}(\theta _7^{\bullet}(q,-t+iw))$, 
    one obtains that

     $$\begin{array}{ccl}
      T_I&=&(3\sqrt{2}q/16)(56q^{27}t^6-1260q^{27}t^4+3402q^{27}t^2-729q^{27}-
      48q^{20}t^5+720q^{20}t^3-972q^{20}t\\ \\
      &&+40q^{14}t^4-360q^{14}t^2+162q^{14}-32q^9t^3+144q^9t+24q^5t^2-
      36q^5-16q^2t+8)~.
    \end{array}$$
    We set $T_{I,k}:=\partial ^kT_I/\partial t^k|_{t=0}$. These functions
    are equal respectively to:

    $$\begin{array}{cclccl}
      T_{I,0}&=&-(3\sqrt{2}/16)q(81q^{18}+18q^9-18q^5+4)(9q^9-2)~,&
      T_{I,4}&=&-90\sqrt{2}q^{15}(63q^{13}-2)~,\\ \\
      T_{I,1}&=&-(3\sqrt{2}/4)q^3(243q^{18}-36q^7+4)~,&
     T_{I,5}&=&-1080\sqrt{2}q^{21}~,\\ \\
      T_{I,2}&=&(9\sqrt{2}/4)q^6(567q^{22}-60q^9+4)~,&
      T_{I,6}&=&7560\sqrt{2}q^{28}~.\\ \\
      T_{I,3}&=&18\sqrt{2}q^{10}(45q^{11}-2)~,&&&
    \end{array}$$
    These derivatives do not change sign on the interval $[0.6,0.75]$, with 
    sgn$(T_{I,k})=(-1)^k$. The Taylor series of $T_I$ at $t=0$ reads:

    $$T_I=(T_{I,0}+tT_{I,1})+(t^2/2)(T_{I,2}+tT_{I,3}/3)+(t^4/24)(T_{I,4}+tT_{I,5}/5)+
    T_{I,6}t^6/6!~.$$
    One verifies directly that
    the following inequalities hold true:

     $$
    T_{I,0}>|T_{I,1}|+0.02~,~~~\, T_{I,2}>|T_{I,3}|/3~~~\, {\rm and}~~~\,
    T_{I,4}>|T_{I,5}|/5~.$$
    Hence the Taylor series of $T_I$ 
    takes only values $>0.02$ for $t\in [0,1]$. We need, however,
    a stronger result. It is to be checked directly that
    
    (i) $T_{I,0}-|T_{I,1}|/2>0.2$, so for $t\in [0,1/2]$, one has
    $T_I>0.2$, $\theta _7^{\bullet}>0.2>0.038>|\theta _{*}^{\bullet}|$ and thus
    $|\theta |>0$;

    (ii) $T_{I,2}/2-|T_{I,3}|/6>0.1$, so for $t\in [1/2,1]$, one has

    $$T_I>(T_{I,0}-|T_{I,1}|)+(1/4)(T_{I,2}/2-|T_{I,3}|/6)>0.02+0.025>0.038$$
    and again $|\theta |\geq \theta _7^{\bullet}-|\theta _{*}^{\bullet}|>0$.
  \end{proof}

\end{document}